\documentclass[11pt,twoside]{article}
\usepackage{amsmath,amssymb,amsthm,euscript,color}

\numberwithin{equation}{section}
\newtheorem{theorem}{Theorem}[section]

\newtheorem{lemma}[theorem]{Lemma}

\newcommand{\myendproof}{ \hfill $\square$}

\newcommand{\bke}[1]{\left ( #1 \right )}
\newcommand{\bkt}[1]{\left [ #1 \right ]}
\newcommand{\bket}[1]{\left \{ #1 \right \}}
\newcommand{\norm}[1]{ \| #1  \|}

\newcommand{\bka}[1]{{\left\langle #1 \right\rangle}}

\newcommand{\abs}[1]{\left | #1 \right |}

\def\al{\alpha}
\def\be{\beta}
\def\ga{\gamma}
\def\de{\delta}

\def\e {\varepsilon}

\def\th{\theta}

\def\la{\lambda}
\def\si{\sigma}
\def\om{\omega}

\def\De{\Delta}

\newcommand{\R}{\mathbb{R}}

\newcommand{\Z}{\mathbb{Z}}
\newcommand{\N}{\mathbb{N}}

\renewcommand{\div}{\mathop{\rm div}}
\newcommand{\curl} {\mathop{\rm curl}}

\newcommand{\esssup} {\mathop{\rm ess\,sup}}
\newcommand{\pd}{\partial}
\newcommand{\nb}{\nabla}
\newcommand{\td}{\tilde}

\newcommand{\wbar}[1]{\overline{\rule{0pt}{2.4mm} {#1}}}
\newcommand{\lec}{{\ \lesssim \ }}

\newcommand{\I}{\infty}
\newcommand{\ot}{\otimes}

\newcommand{\LL}{local-Leray }

\newcommand{\SVERAK}{{\v Sver\'ak}}
\newcommand{\bs}{\backslash}
\newcommand{\um}{u_{\text{mild}}}

\begin{document}

\title{Forward Discretely Self-Similar Solutions
of the Navier-Stokes Equations}
\author{Tai-Peng Tsai}
\date{}
\maketitle

{\bf Abstract}. Extending the work of Jia and \SVERAK{} on
self-similar solutions of the Navier-Stokes equations, we show the
existence of large, forward, discretely self-similar solutions.

\section{Introduction}
Denote $\R^4_+ = \R^3 \times (0,\I) $.  Consider the 3D incompressible
Navier-Stokes equations for velocity $u:\R^4_+ \to \R^3$ and pressure
$p: \R^4_+ \to \R$,
\begin{equation}
\label{NS1} \pd _t u -\De u + (u \cdot \nb) u  + \nb p =0 , \quad
\div u = 0,
\end{equation}
in $\R^4_+$, coupled with the initial condition
\begin{equation}
\label{NS2} u|_{t=0}  = u_0, \quad \div u_0=0.
\end{equation}

The system \eqref{NS1} enjoys a scaling property: If
$u(x,t)$ is a solution, then so is
\begin{equation}
\label{scaling}
u^{(\la)}(x,t):= \la u(\la x, \la^2 t)
\end{equation}
for any $\la>0$. We say $u(x,t)$ is {\bf self-similar} (SS) if
$u=u^{(\la)}$ for every $\la>0$. In that case, the value of $u(x,t)$
is decided by its value at any time moment $t=\frac 1{2a}$ and
\begin{equation}
\label{SS.formula}
u(x,t)= \la(t) U (\la(t)x), \quad \la(t) = \frac 1{\sqrt{2at}},
\end{equation}
where $U(x) = u(x,\frac 1{2a})$ and $a>0$ is a parameter.  On the
other hand, if $u=u^{(\la)}$ only for one particular $\la>1$, we say
$u$ is {\bf discretely self-similar} (DSS) with {\bf factor} $\la$, or
{\bf $\la$-DSS}.  Its value in $\R^4_+ $ is decided by its value in
the strip $x\in \R^3$ and $1\le t < \la^2$.  We consider being SS as a
special case of being DSS, and would say ``{\it strictly DSS}'' to exclude
the former.  We call them {\bf forward} to indicate they are defined
for $t>0$.  We can also consider \eqref{NS1} for
\begin{enumerate}
\item $(x,t) \in \R^3 \times (-\I,0)$, or
\item $x \in \R^3$, $u=u(x)$ is time independent.
\end{enumerate}
For both cases the scaling law \eqref{scaling} still holds, and we
define {\bf backward} and {\bf stationary} SS and DSS solutions in the
same manner. In particular, a backward SS solution satisfies
\eqref{SS.formula} with $a<0$, a stationary SS solution satisfies
\begin{equation}
u(x)= \la(x) U (\la(x)x), \quad \la(x) = \frac 1{|x|},
\end{equation}
with $U(x) = u(x)$, and the profile $U(x)$ of the SS
solution for all three cases satisfies Leray's equations
\begin{equation}
\label{eq1.6}
-\De U - aU - ax \cdot \nb U + (U \cdot \nb) U  + \nb p =0 , \quad
\div U = 0,
\end{equation}
with $a>0$, $a<0$ and $a=0$ respectively.  Note the stationary SS
solutions are often called minus-one homogeneous solutions in the
literature.

When $u(x,t)$ is either SS or DSS, then so is $u_0(x)$.
Thus it is natural to assume
\begin{equation}
\label{u0-est}
|u_0(x)| \le \frac {C_*}{|x|} ,\quad 0 \not = x \in \R^3
\end{equation}
for some constant $C_*>0$ and look for solutions satisfying
\begin{equation}
\label{eq1.8}
|u(x,t)| \le \frac {C(C_*)}{|x|}, \quad \text{or }\quad
\norm{u(\cdot,t)}_{L^{3,\I}}\le C(C_*).
\end{equation}
Here by $L^{q,r}$, $1\le q,r\le \I$, we denote the Lorentz spaces.  In
such classes, with sufficiently small $C_*$, the unique existence of
mild solutions -- solutions of the integral equation version of
\eqref{NS1}--\eqref{NS2} via contraction mapping argument, see
\eqref{mild-def} -- has been obtained by Giga-Miyakawa \cite{GM} and
refined by Cannone-Meyer-Planchon \cite{CMP, CP}.  As a consequence,
if $u_0(x)$ is SS or DSS satisfying \eqref{u0-est} with small $C_*$
and $u(x,t)$ is a corresponding solution satisfying \eqref{eq1.8} with
small $C(C_*)$, the uniqueness property ensures that $u(x,t)$ is also
SS or DSS, because $u^{(\la)}$ is another solution with same bound and
same initial data $u_0^{(\la)}=u_0$.  For large $C_*$, mild solutions
still make sense but there is no existence theory since perturbative
methods like the contraction mapping no longer work.

Alternatively, one may try to extend the concept of weak solutions
(which requires $u_0 \in L^2(\R^3)$) to more general initial data. One
such theory is \LL solutions, constructed by Lemari\'e-Rieusset
\cite{LR} (to be defined below). However, there is no uniqueness
theorem for them and hence the existence of large SS or DSS solutions
was unknown.

In a surprising recent preprint \cite{JS}, Jia and \SVERAK{}
constructed SS solutions for every SS $u_0$ which is locally H\"older
continuous. Their main tool is a local H\"older estimate of the
solution near $t=0$, assuming minimal control of the initial data in
the large (see Theorem \ref{th2.3}). This estimate enables them to
prove a priori estimates of SS solutions, and then show their
existence by applying the Leray-Schauder degree theorem.  Note that
this existence theorem does not assert uniqueness. In fact,
non-uniqueness is conjectured in \cite{JS}.

In this note, as an attempt to understand \cite{JS}, we consider the
existence of discretely self-similar solutions for DSS $u_0$ satisfying
\eqref{u0-est} with large $C_*$.

We now recall the definition of \LL solutions, see \cite{LR, JS}.
Suppose
\begin{equation}
\label{u0-cond}  
u_0 \in L^2_{loc}(\R^3;\R^3), \quad \norm{u_0}_{L^2_{uloc}}:=
 \sup_{x_0 \in \R^3} \bigg(\int_{B_1(x_0)} |u_0|^2(x) dx \bigg)^{\frac 12} <
\I, \quad \div u_0 = 0.
\end{equation}
A vector field $u \in L^2_{loc}(\R^3 \times [0,\I))$ is called a {\bf
    \LL solution} of \eqref{NS1}--\eqref{NS2} if (i)
\begin{equation}
\esssup_{0\le t <R^2} \sup _{x_0\in \R^3} \int_{B(x_0,R)} |u(x,t)|^2 dx+
\sup _{x_0\in \R^3} \int_0^{R^2}\int_{B(x_0,R)} |\nabla u(x,t)|^2 dx\,dt<\I,
\end{equation}
%\sidenote{weaker than $u \in L^\I(\R^+,E_2)$}
\begin{equation}
\label{decay-cond}
\text{ and }\quad \lim_{|x_0|\to \I}  \int_0^{R^2}\int_{B(x_0,R)} |u(x,t)|^2 dx\,dt=0,
\end{equation}
for any $R<\I$, (ii) together with some distribution $p$ in $\R^4_+$
they satisfy \eqref{NS1} in $\R^4_+$ in the sense of distributions,
(iii) $\lim_{t \to 0+} \norm{u(\cdot,t)-u_0}_{L^2(K)}=0$ for any
compact set $K \subset \R^3$, and (iv) $u$ is {\it suitable}, i.e., it
satisfies the {\it local energy inequality} in the sense of
Caffarelli, Kohn, and Nirenberg \cite{CKN}.

The class of \LL solutions contains both Leray-Hopf weak solutions
(with $u_0 \in L^2(\R^3)$) and mild solutions (with $u_0$ in
$L^3(\R^3)$ or $VMO^{-1}$), and is strictly larger.  It is useful for
our purpose because it allows initial data of the size $|u_0(x)| \sim
\frac C{|x|}$, because the local energy inequality is valid, and
because a priori local energy estimates are available (see Lemma
\ref{th2.1}).

We now state our main theorems on the existence of forward
discretely self-similar solutions.  We first consider those with DSS
factor close to one.  We denote $\bka{z}=(|z|^2+2)^{1/2}$ for $z \in
\R^n$, $n\in\N$.

\begin{theorem}[Existence of DSS solutions with factor close to one]
\label{th1.1}\

For any $0<\ga<1$ and $C_*>0$, there is $\la_*=\la_*(\ga,C_*)\in
(1,2)$ such that the following hold.  Suppose $u_0\in
C^{\ga}_{loc}(\R^3 \bs \{ 0 \})$, $\norm{u_0}_{C^\ga(\wbar B_2 \bs
  B_1)} \le C_*$, $\div u_0=0$, and $u_0$ is DSS with factor $\la\in
(1,\la_*]$.  Then there is a \LL solution $u$ of \eqref{NS1} with
initial data $u_0$ that is DSS with factor $\la$ and, for $v(\cdot,t):= 
u(\cdot,t)-e^{t\De}u_0$
\begin{equation}
\label{eq1.12}
|u(x,t)|\le \frac{C}{ |x|+{\sqrt t}}, \quad
|v(x,t)|\le \frac{C\sqrt t}{ |x|^2+ t}
\end{equation}
in $\R^4_+$ with $C=(\ga,C_*)$.  It is also a mild solution in the class
\eqref{eq1.12}$_1$. If furthermore,
$\norm{u_0}_{C^{1,\be}(\wbar B_2 \bs B_1)} \le C_*$ for some
$0<\be<1$, then
\begin{equation}\label{eq1.13}
|v(x,t)| \le \frac{C}{\sqrt t} \bka{\frac x{\sqrt t}}^{-3} \log \bka{\frac
{x}{\sqrt t}} , \quad |D_xv(x,t)| \ \le \frac{C}{ t} \bka{\frac
  x{\sqrt t}}^{-3} 
\end{equation}
in $\R^4_+$ with $C=(\be,C_*)$.
\end{theorem}

Note that $\la-1>0$ has to be small enough.

A similar result is true for axisymmetric initial data with no swirl
that is DSS with arbitrary factor. We recall that a vector field $u$
in $\R^3$ is called {\bf axisymmetric} if in cylindrical coordinates
$r,\th,z$ with $(x_1,x_2,x_3)=(r\cos \th,r\sin \th,z)$ and
$r=\sqrt{x_1^2+x_2^2}$, it is of the form
\begin{equation} %\label{v-cyl}
u(x) = u^r(r,z)e_r + u^\th(r,z)e_\th + u^z(r,z)e_z.
\end{equation}
The components $u^r,u^\th,u^z$ do not depend upon $\th$ and the basis
vectors $e_r,e_\th,e_z$ are
\begin{equation}
e_r = \left(\frac {x_1}r,\frac {x_2}r,0\right),\quad
e_\th = \left(-\frac {x_2}r,\frac {x_1}r,0\right),\quad
e_z = (0,0,1).
\end{equation}
It is called ``no swirl'' if $u^\th=0$.  This class of vector fields
is preserved under \eqref{NS1}. If the initial data $u_0\in
H^2(\R^3)$ is axisymmetric with no swirl, global in-time regularity of
the solution was proved independently by Ukhovskii-Yudovich \cite{UY}
and Ladyzhenskaya \cite{Lad}. See \cite{LMNP} for a refined proof.
The case of general axisymmetric flow with $u^\th\not =0$ is open.

\begin{theorem}
[Existence of axisymmetric DSS solutions with no swirl]
\label{th1.2} \

For any $1<\la< \I$, $0<\ga<1$,  and $C_*>0$,  suppose $u_0\in
C^{\ga}_{loc}(\R^3 \bs \{ 0 \})$, is axisymmetric with no swirl, DSS
with factor $\la$, $\div u_0=0$, and $\norm{u_0}_{C^\ga(\wbar B_{\la}
  \bs B_1)} \le C_*$.  Then there is a \LL solution $u$ of \eqref{NS1}
with initial data $u_0$ that is DSS with factor $\la$, axisymmetric
with no swirl, and satisfies \eqref{eq1.12} in $\R^4_+$ with $C=
C(\la,\ga,C_*)$. It is also a mild solution in the class
\eqref{eq1.12}$_1$.

If furthermore, $\norm{u_0}_{C^{1,\be}(\wbar B_{\la} \bs B_1)} \le
C_*$ for some $0<\be<1$, then it satisfies \eqref{eq1.13} in $\R^4_+$
with $C= C(\la,\be,C_*)$.
\end{theorem}

If $C_*$ is small, the existence is known by \cite{GM,CMP, CP}.
Theorems \ref{th1.1} and \ref{th1.2} are concerned with large $C_*$.

If one assumes higher regularity of $u_0$, say $u_0 \in
C^{3,\be}_{loc}(\R^3 \bs \{ 0 \})$ for some $0<\be<1$, the same proof
of \cite[Th 4.1]{JS} shows
\begin{equation}
\label{th1.1:eq3}
|v(x,t)|\le \frac{Ct }{ (|x|+{\sqrt t})^3}, \qquad (\frac {|x|}{\sqrt
  t} >C).
\end{equation}
This rate is optimal in view of the explicit spatial asymptotes for
small SS solutions with smooth initial data in
\cite{Bra}. Eq.~\eqref{eq1.13} is slightly worse than
\eqref{th1.1:eq3} by a log factor, but only assumes $u_0 \in
C^{1,\be}_{loc}$.  There is a gap between \eqref{eq1.12} and
\eqref{eq1.13} especially when one takes $1-\ga = \be \ll 1$. It is
probably because we require pointwise bound of the source term when we
estimate the Stokes system. One may be able to narrow the gap by
considering integral bounds of the source term.

Our approach follows that of \cite{JS}, and relies heavily on the a
priori estimates of the solutions, see Lemmas \ref{th3.1} and
\ref{th3.2}. One difference is that, instead of estimating a
stationary solution of \eqref{eq1.6}, we need to estimate a time
dependent solution of \eqref{NS1}.  Another difference is the
following: \cite{JS} first proves a priori estimates and constructs
solutions for smooth initial data, and then gets solutions for $C^\ga$
data by approximation. In contrast, we prove a priori estimates and
construct solutions for $C^\ga$ initial data directly.  The reason for
this change is that, at least for Theorem \ref{th1.1}, we need the
explicit dependence of $\la_*$ on the local $C^\ga$-norm of the data.

To extend these results, one may look for DSS solutions with 
DSS initial data of the form
\begin{equation}
u_0 = u_0^1 + u_0^2
\end{equation}
where $u_0^1$ is SS and large, while $u_0^2$ is DSS with a large
factor and is sufficiently small. When $\la$ is large, a priori estimates
seem unavailable, and one may try to study the linearized flow around
$u^1$, a SS solution with initial data $u_0^1$. It turns out to be
very challenging.

Another interesting problem is the non-uniqueness of \LL solutions
conjectured by \cite{JS} and \SVERAK{} \cite{Sve-video}: Considers SS
solutions $W_\si$ with SS initial data $\si u_0$, $\si>0$. For $\si$
small, $W_{\si}$ is unique (see Lemma \ref{th3.2}). However, when one
increases $\si$, one might get bifurcation. If the bifurcation is of
saddle-node type, we get two SS solutions with the same initial data.
If it is a Hopf bifurcation, the new solutions would be time periodic
in the similarity variables ($y=t^{-1/2}x$ and $s=\log t$) and
correspond to DSS solutions.  These kind of DSS solutions $u$ are
different from those in Theorems \ref{th1.1} and \ref{th1.2} since
their initial data $u_0$ are SS and only the difference
$v(\cdot,t)=u(\cdot,t)- e^{t\De}u_0$ are strictly DSS. One may
approximate $u_0$ by DSS data $u_0^\e$ and take limits $\e\to 0$, and
try to show that the strict DSS property of the corresponding
solutions $u^\e$ is somehow preserved in the limit. Of course this is
purely speculation.

The existence question of discretely self-similar solutions also occur
in two other instances for Navier-Stokes equations: (i) For singular
backward solutions $u(x,t): \R^3 \times (-\I,0) \to \R^3$ of
\eqref{NS1}, the nonexistence of SS solutions under some minimal
integrability assumptions was proved in \cite{NRS} and
\cite{Tsai98}. The existence problem for DSS solutions under the same
integrability assumptions is open; (ii) For stationary Navier-Stokes
equations, the self-similar (minus-one homogeneous) solutions in $\R^3
\backslash \{ 0 \}$ are shown by \SVERAK{} \cite{Sve} to be exactly
those axisymmetric solutions found by Landau \cite{Landau,LL}. The
existence problem for strictly DSS solutions is open.

The rest of this note is structured as follows: In \S2 we 
consider the Stokes system. In \S3 we prove a priori estimates
for DSS \LL solutions. In \S 4 we show their uniqueness for small
initial data. Finally in \S 5 we prove their existence.

{\it Notation}. We denote $\bka{z}=(|z|^2+2)^{1/2}$ for $z \in \R^n$,
$n\in\N$, and $A \lec B$ if there is a constant $C$, which may change
from line to line, such that $A \le CB$. We denote by $D^k_x u$ all
$k$-th order partial derivatives of $u$ with respect to the variable
$x$.
%=====================================================================
\section{Stokes system}
Consider the non-stationary Stokes system in $\R^3$
with a force tensor $f=(f_{ij})$
\begin{equation}
\label{v-eq1}
\pd_t v - \De v + \nb p=\nb \cdot f, \quad \div v=0,
\quad
v|_{t=0}=0.
\end{equation}
Here $(\nb \cdot f)_j = \sum_k \pd_k f_{kj}$.
If $f$ has sufficient decay, a solution is given by $v=\Phi f$, with
\begin{align}
\label{La-def} (\Phi f)_i(x,t) &
=\int_0^{t}\int_{\R^3}\pd_{x_k}S_{ij}(x-y,t-s) f_{kj}(y,s)dyds
\end{align}
and $S=(S_{ij})$, the Oseen tensor, is
the fundamental solution of the non-stationary Stokes
system in $\R^3$ (see \cite{Oseen} and \cite[page 27]{Solonnikov})
\begin{equation}
\label{Stokes-tensor}
S_{ij}(x,t)=\Gamma(x,t)\delta_{ij}+\frac{1}{4\pi}\frac{\partial^2}{\partial
x_i\partial x_j}\int_{\R^3}\frac{\Gamma(y,t)}{\abs{x-y}}dy,\quad
Q_j(x,t)=\frac{\delta(t)}{4\pi}\frac{x_j}{\abs{x}^3},
\end{equation}
where $\Gamma(x,t)=(4\pi t)^{-n/2}e^{-x^2/4t}$ is the heat kernel. It
is known in \cite[Theorem 1]{Solonnikov} that %the %Oseen
%tensor
%$S=(S_{ij})$ satisfies the following estimates:
%
\begin{equation}\label{estimate-T}
\abs{D^\ell_x\pd^k_t S(t,x)}\leq C_{k,l}
(|x|+\sqrt{t})^{-3-\ell-2k}, \quad (\ell,k \ge 0),
\end{equation}
where $D^\ell_x$ indicates $\ell$-th order derivatives with respect
to the variable $x$.

When the bilinear operator $\Phi(u\ot v): X \times X \to X$ is
well-defined on some Banach space $X$ of $\R^3$-valued fields on
$\R^4_+$, we say $u(x,t)$ is a {\bf mild solution} of \eqref{NS1} and
\eqref{NS2} if
\begin{equation}\label{mild-def}
u(\cdot,t) = e^{t\De}u_0 - \Phi (u \ot u) (\cdot,t) \quad \text{in }X.
\end{equation}

We start with an integral estimate.

\begin{lemma} \label{th:cal}
Let $0<a<5$, $0<b<5$ and $a+b>3$. Then
\begin{equation}
\phi(x,a,b)=\int_0^1 \int_{\R^3} (|x-y|+\sqrt{1-t})^{-a}
(|y|+\sqrt t)^{-b} dy\,dt
\end{equation}
is well defined for $x\in \R^3$  and
\begin{equation}
\phi(x,a,b) \lec R^{-a}+R^{-b} + R^{3-a-b} [1+ (1_{a=3}+1_{b=3}) \log R]
\end{equation}
where $R=|x|+2$.
\end{lemma}
The cases $a,b \in \{ 3,4\}$ are stated in \cite[(4.12)]{JS}. We will
also use $ 4 \le a < 5$ and $2 \le b <3$.

\begin{proof}
Clearly $\phi \lec 1$ for $R\le 8$. Consider now $R>8$. Estimate the
integral in 3 parts:
\begin{equation}
\int_0^1 \int_{|y|>2R} (\cdot)\lec \int_0^1 \int_{|y|>2R}
|y|^{-a-b}dy\,dt = C R^{3-a-b},
\end{equation}
\begin{equation}
  \begin{split}
&\int_0^1 \int_{|y|<R/2}(\cdot) \lec \int_0^1 \int_{|y|<R/2}R^{-a}
 (|y|+\sqrt t)^{-b}dy dt
\\
&\lec R^{-a}  \bke{\int_0^1  \int_{|y|<1}+\int_0^1 \int_{1<|y|<R/2}} (\cdot)\lec
R^{-a} (1+R^{3-b}(1+ 1_{b=3} \log R )),
  \end{split}
\end{equation}
and similarly
\begin{equation}
  \begin{split}
&\int_0^1 \int_{R/2<|y|<2R} (\cdot) \lec \int_0^1 \int_{R/2<|y|<2R}
(|x-y|+\sqrt {1-t})^{-a} R^{-b} dydt
\\
&\lec R^{-b} \int_0^1 \int_{|z|<3R}
(|z|+\sqrt {t})^{-a}  dzdt
 \lec R^{-b}  (1+R^{3-a}(1+ 1_{a=3} \log R )).
  \end{split}
\end{equation}
\end{proof}

% ----------------------------------------------------------------------
\begin{lemma} \label{th:Stokes}
Suppose $|f(x,t)| \le \frac 1t \bke{\frac{\sqrt t}{|x|+\sqrt t}}^{2+m}$ in
$\R^4_+$ for $0\le m<1$. Then
\begin{equation}
\label{eq2.13}
|\Phi f(x,t)| \lec \frac1{\sqrt t} \bke{\frac{\sqrt t}{|x|+\sqrt
    t}}^{2+m}, \quad \forall (x,t)\in \R^4_+.
\end{equation}

\end{lemma}
\begin{proof}
By \eqref{estimate-T} and change of variables $x=\sqrt t \td x$,
$y=\sqrt t \td y$, $s=t \td s$,
\begin{equation}
| \Phi f(x,t)| \lec \int_0^t \int_{\R^3}
 (|x-y|+\sqrt{t-s})^{-4 }
s^{m/2}(|y|+\sqrt s)^{-2-m} dy\,ds
\end{equation}
\begin{equation}
= t^{-1/2}  \int_0^1 \int_{\R^3}
 (|\td x-y|+\sqrt{1-s})^{-4 }s^{m/2}
(|y|+\sqrt s)^{-2-m} dy\,ds .
\end{equation}
By $s^{m/2}\le 1$ and Lemma \ref{th:cal}, we get $ | \Phi f(x,t)| \lec
 t^{-1/2} \bka{\td x} ^{-(2 +m)}$, i.e., \eqref{eq2.13}.
\end{proof}

% ----------------------------------------------------------------------
We now show H\"older estimates in space and time. Fix $0<\th < 1$.
Denote a local parabolic H\"older estimate for $(x,t)\in \R^4_+$:
\begin{equation}
[u]_\th (x,t) := \sup_{\td x,\td t}\bket{ \frac{|u(x,t)-u(\td x, \td
    t)|}{\de^\th} \Bigg|\quad \de:= |x-\td x| + \sqrt{|t-\td t|} \le
  \frac{\sqrt t}{10} }.
\end{equation}

\begin{lemma} \label{th:Holder}
Suppose $|f(x,t)| \le (|x|+\sqrt t)^{-2}$ in $\R^4_+$. Then $ \Phi f$
is locally H\"older continuous in $x$ and $t$ with any exponent
$0<\th<1$ and for any $T \in (1,\I)$
\begin{equation}
[  \Phi f]_\th (x,t) \le C_T \bka{x}^{-2}, \quad
\forall x\in \R^3, \quad \forall 1\le t \le T.
\end{equation}
\end{lemma}

\begin{proof}
We may assume $t=1$.

We first show spatial H\"older estimate.
For $h\in \R^3$ with $\de=|h|<0.1$, we have
\begin{equation}
|  \Phi f(x+h,t)-   \Phi f(x,t)  |\le I_1 + I_2
\end{equation}
\begin{equation}
:=
\int_0^1 \bke{ \int_{|z|>2\de} + \int_{|z|<2\de} }
| D_x S(z+h,s) -  D_x S(z,s) |\cdot |f(x-z,1-s)|\,dz\,ds
\end{equation}

For $I_1$, by mean value theorem and \eqref{estimate-T},
\begin{equation}
I_1\le \int_0^1\int_{|z|>2\de} |h|
|D^{2}_x S(z,s)  |\cdot |f(x-z,1-s)|\,dz\,ds
\end{equation}
\begin{equation}
\le \int_0^1\int_{|z|>2\de} \de^\th (|z|+\sqrt s)^{-4 - \th}
(|x-z|+\sqrt {1-s})^{-2} \,dz\,ds
\end{equation}
By Lemma \ref{th:cal},
\begin{equation}
\label{I1.est}
I_1 \lec \de^\th \bka{x}^{-2}.
\end{equation}

For $I_2$, we have $|z+h|<3\de$.
If $|x|<4\de$, splitting $0<t<1$ to $0<t<\frac 12$ and
$\frac 12 <t<1$ and using \eqref{estimate-T}, we have
\begin{equation}
I_2 \lec \int_0^{1/2} \int _{|z|<3\de} (|z|+\sqrt s)^{-4}
 \,dz\,ds+\int_{1/2}^1 \int _{|z|<7\de}
(|z|+\sqrt {1-s})^{-2} \,dz\,ds
\end{equation}
Using
\begin{equation}
\label{eq2.27}
\int_0^1 \int_{|z|<\de} (|z|+\sqrt s)^{-\al} dz\, ds \lec \left
\{ \begin{array}{c} \de^{5-\al}, \quad (2<\al<5),\\ \de^3 \log (1/\de),
  \quad (\al=2),\end{array} \right .
\end{equation}
(which can be shown by splitting $(0,1)=(0,\de^2) \cup [\de^2,1)$),
we get
\begin{equation}
I_2 \lec \de + \de^{3}\log(1/\de) \lec  \de.
\end{equation}

If $4\de<|x|$,
we have (using \eqref{eq2.27} again)
\begin{equation}
I_2 \lec \int_0^1 \int _{|z|<3\de} (|z|+\sqrt s)^{-4} (|x|+\sqrt {1-s})^{-2}
\,dy\,ds
\end{equation}
\begin{equation}
\lec
\int_0^{1/2} \int _{|z|<3\de} (|z|+\sqrt s)^{-4}\bka{x}^{-2}
 \,dz\,ds+\int_{1/2}^1 \int _{|z|<3\de}
(|x|^2+1-s)^{-1} \,dz\,ds
\end{equation}
\begin{equation}
\lec  \de\bka{x}^{-2} + \de^3 \log \frac {1+|x|^2}{|x|^2} \lec 
\de\bka{x}^{-2},
\end{equation}
which is also bounded by the right side of \eqref{I1.est} (and much
less if $\de \ll 1$).

% ------------------------------------------------------------------
We next show temporal H\"older estimate.  Take $\tau=\de^2$ with
$0<\de< 0.1$. (For $\tau<0$ we can reverse $t$ and $t+\tau$).  We have
\begin{equation}
|  \Phi f(x,1+\tau)-   \Phi f(x,1)  |\le I_1 + I_2+I_3
\end{equation}
\begin{equation}
:=
\int_0^{1-\tau} \int
|D  S(z,1+\tau-s) - D  S(z,1-s) |\cdot |f(x-z,s)|\,dz\,ds
\end{equation}
\begin{equation}
+ \int_{1-\tau}^{1+\tau} \int | D S(z,1+\tau-s) |\cdot |f(x-z,s)|\,dz\,ds
\end{equation}
\begin{equation}
+ \int_{1-\tau}^{1} \int |- D S(z,1-s) |\cdot |f(x-z,s)|\,dz\,ds.
\end{equation}

For $I_1$, by mean value theorem and \eqref{estimate-T},
\begin{equation}
I_1\lec \int_0^{1-\tau} \int \tau (|z|+\sqrt{1- s})^{-6} (|x-z|+\sqrt s)^{-2}\,dz\,ds
\end{equation}
\begin{equation}
\le \int_0^{1-\tau} \int \de^\th (|z|+\sqrt{1- s})^{-4-\th} (|x-z|+\sqrt s)^{-2}\,dz\,ds.
\end{equation}
By Lemma \ref{th:cal},
\begin{equation}
I_1 \lec \de^\th \bka{x}^{-2}.
\end{equation}
The two terms $I_2$ and $I_3$ are similar and it suffices to estimate
$I_3$: By \eqref{estimate-T},
\begin{equation}
I_3\lec \int_{1-\tau}^{1} \int (|z|+\sqrt{1- s})^{-4} \,(|x-z|+1)^{-2}
\,dz\,ds.
\end{equation}
Integrating in $s$ first, we get
\begin{equation}
I_3 \lec \int_{\R^3} \frac {\tau}{|z|^2+\tau} |z|^{-2}\,
(|x-z|+1)^{-2}\,dz.
\end{equation}
If $|x|\le 2$, then $|x-z|+1\sim \bka{z}$ and
\begin{equation}
I_3 \lec  \int \frac {\de^\th}{|z|^\th} |z|^{-2}\,  \bka{z}^{-2}\,dz \lec \de^\th.
\end{equation}
If $|x|>2 $,
\begin{equation}
I_3 \lec \int _{|z|<|x|/2}  \frac {\tau}{|z|^2+\tau} |z|^{-2}\, |x|^{-2}\,dz
+ \int _{|z|>|x|/2}  {\tau}|z|^{-4}\, |x-z|^{-2}\,dz
\end{equation}
\begin{equation}
=C \de |x|^{-2} + C'\de^2 |x|^{-3} \lec \de |x|^{-2}.
\end{equation}
The equality here is obtained by rescaling.
\end{proof}

Finally we give a Liouville lemma.

\begin{lemma}
\label{th2.4-new}
If $v(x,t):\R^4_+ \to \R^3$ satisfies $|v(x,t)| \le C t^{-1/2}
\bka{x/\sqrt t}^{-1-\ga}$ for some $0<\ga<1$ and
\begin{equation}
\pd_t v - \De v + \nb p = 0, \quad \div v=0, \quad v|_{t=0}=0  ,
\end{equation}
for some distribution $p$, then $v \equiv 0$.
\end{lemma}

It is similar to \cite[Lemma 4.1  (i)]{JS}, with exactly 
the same proof.
%====================================================================

\section{A priori estimates for DSS solutions}

We first recall a couple estimates for \LL solutions from \cite{JS}.

\begin{lemma}[\cite{JS} Lemma 3.1]\label{th2.1}
There are constants $0<C_1< 1<C_2$ such that the following holds.
Suppose $\div u_0=0$, $A = \sup_{x_0\in\R^3}\int_{B_R(x_0)} |u_0(x)|^2
dx<\I$ for some $R>0$ and $u$ is a \LL solution with initial data
$u_0$. Then for $\la = {C_1}\min \bke{A^{-2}R^2,1}$,
\begin{equation}
\esssup_{0<t<\la R^2} \sup_{x_0\in\R^3}\int_{B_R(x_0)} |u(x,t)|^2 dx +
\sup _{x_0\in\R^3}\int_0^{\la R^2}\int_{B_R(x_0)} |\nb u(x,t)|^2 dx dt\le C_2A.
\end{equation}
Because \eqref{decay-cond} holds for $u$, for some
$p(t)=p_{x_0,R}(t)$,
\begin{equation}
\label{p-est0}
\sup _{x_0\in\R^3}\int_0^{\la R^2} \int_{B_R(x_0)} |p(x,t)-p(t)|^{3/2}
dx dt\le C_2A^{3/2} R^{1/2}.
\end{equation}
\end{lemma}

\begin{theorem}[\cite{JS} Th 3.2]\label{th2.3}
Suppose $\div u_0=0$, $\norm{u_0 }_{L^2_{uloc} }^2\le A<\I$, and
$\norm{u_0}_{C^\ga(B_2)}\le M <\I$ for some $\ga\in(0,1)$. Then
there exists $T=T(A,\ga,M)>0$ such that any \LL solution $u$ with
initial data $u_0$ satisfies $u \in
C^\ga_{par}(\wbar{B_{1/4}}\times [0,T])$
  and
\begin{equation}
\norm{u}_{C^\ga_{par}(\wbar{B_{1/4}}\times [0,T])}\le C(A,\ga,M).
\end{equation}
\end{theorem}

Note that in the proof of \cite[Th 3.1]{JS}, which is needed for \cite[Th 3.2]{JS}, that the pressure is in $L^{5/3}$ in time does not follow from the elliptic equation
it satisfies. Rather, it follows from \cite[(3.3)]{JS} and that the velocity is in
$L^{10/3}_t$.

We now consider DSS solutions with factor close to one.

\begin{lemma}[A priori estimates for DSS solutions with factor close to one]
\label{th3.1}\

(i) For any $0<\ga<1$ and $C_*>0$, there is $\la_*=\la_*(\ga,C_*)\in
(1,2)$ such that the following hold. Suppose $1<\la\le \la_*$ and $u$
is a forward $\la$-DSS \LL solution of \eqref{NS1} with $\la$-DSS
initial data $u_0\in C^{\ga}_{loc}(\R^3 \bs \{ 0 \})$ satisfying $\div
u_0=0$ and $\norm{u_0}_{C^\ga(\wbar B_2 \bs B_1)} \le C_*$.  Then
$v(x,t):=u(x,t) - (e^{t \De} u_0)(x)$ satisfies, for some
$C=C(\ga,C_*)$,
\begin{equation}
\label{eq3.1}\abs{ u(x,t) } <\frac{C}{|x|+\sqrt t}
, \quad \abs{ v(x,t) } <\frac{C}{\sqrt t} \bka{\frac x{\sqrt t}}^{-2},
\quad (x,t)\in \R^4_+.
\end{equation}
Moreover, $u$ is a mild solution of \eqref{NS1} and \eqref{NS2}.

(ii) If  furthermore $\norm{u_0}_{C^{1,\be}(\wbar B_2 \bs B_1)} \le C_*$ for some
$0<\be<1$, then
\begin{equation}\label{eq3.3}
|v(x,t)| \le \frac{C}{\sqrt t} \bka{\frac x{\sqrt t}}^{-3} \log
\bka{\frac {x}{\sqrt t}},  \quad |D_xv(x,t)| \ \le \frac{C}{ t}
\bka{\frac x{\sqrt t}}^{-3} 
\end{equation}
in $\R^4_+$ for some $C=C(\be,C_*)$.
\end{lemma}

\begin{proof}
(i) We will first show a weaker estimate
\begin{equation}\label{eq3.5}
\abs{ v(x,t) } <\frac{C(\ga,C_*)}{\sqrt t}
\bka{\frac x{\sqrt t}}^{-1-\ga}, \quad (x,t)\in \R^4_+.
\end{equation}

We first consider the region below the paraboloid,
\begin{equation}
\label{sub-par}
t\le |x|^2/ R_1^2, \quad (x,t)\in \R^4_+,
\end{equation}
for some $R_1=R_1(\ga,C_*)>0$ sufficiently large to be decided later.
By Theorem \ref{th2.3}, there exists $T_1=T_1(\ga,C_*)>0$ such that for
any $x_0\in \R^3$ with $1 \le |x_0|\le \la$
\begin{equation}
\norm{u,v}^\ga_{par}(\wbar {B_{1/9}(x_0)} \times [0,T_1]) \le C(\ga,C_*),
\end{equation}
where we have used that $e^{t\De}u_0$ satisfies the same H\"older
estimate. Since $v(x,0)=0$, we get
\begin{equation}
%\label{eq3.7}
|v(x,t)| \le C(\ga,C_*)t^{\ga/2}, \quad \frac 89 \le |x| \le \la+\frac 19,
\quad 0 \le t \le T_1.
\end{equation}
Since $u^{(\la^k)}$, $k\in \Z$, is another \LL solution with same
initial data $u_0$, the above estimate remains valid with $v$ replaced
by $v^{(\la^k)}$. Scaling back,
\begin{equation}\label{eq3.8}
| v(x,t)| < \frac{C(\ga,C_*)\,t^{\ga/2}}{(\sqrt{t}+|x|)^{1+\ga}},
\end{equation}
in cylinder $C_k=\{x: \frac 89 \la^k \le |x| \le \la^{k+1} \} \times
[0,\la^{2k}T_1]$ for every $k\in \Z$.  Since $ \frac 89 <\la < 2$, the
union $\cup_{k\in \Z}C_k$ contains a set of the form in
\eqref{sub-par} with $R_1=2 T_1^{-1/2}$.

For the complement, by rescaling it suffices to prove an upper bound
in the region
\begin{equation}\label{inner-region}
t>  |x|^2/R_1^2, \quad 1 \le t\le \la^2.
\end{equation}
This region satisfies $|x|< \la R_1 $. Let $T=4$. (We will take
$T=4\la^2$ for the proof of Lemma \ref{th3.2}.) By Lemma \ref{th2.1},
$u \in (L^\I_t L^2_x \cap L^2_t H^1_x )\subset L^{10/3}_{t,x} $ in
$B_{\la R_1} \times (0,T)$ with the norm bounded by $
C(\ga,C_*)$. Together with the decay for $e^{t\De}u_0$ and estimate
\eqref{eq3.8} for $v$ in the region \eqref{sub-par}, we get $
\norm{u}_{ L^{10/3}_{t,x}(\R^3 \times (0,T))} \le C(\ga,C_*)$.
Since $p = (-\De)^{-1} \pd_i\pd_j u_i u_j$, we get
\begin{equation}
\label{p.global}
\norm{p}_{ L^{5/3}_{t,x}(\R^3 \times (0,T))} \le C(\ga,C_*).
\end{equation}
Denote shrinking annuli $A_k = \{ x\in \R^3: 2\la R_1 +k< |x| < 2\la R_1+20-k\}$
for $k =1,2,\ldots$.  Note that $u \in L^\I(A_1 \times [\frac 12,T])$
since the region is contained in the region \eqref{sub-par}. By
regularity theory for \eqref{NS1}, $D^\ell_x u \in L^\I(A_2 \times
[1,T])$ for $\ell \le 3$. By $-\De p = (\pd_iu_j)(\pd_j u_i) $ and
\eqref{p.global}, we get $D^\ell p \in L^\I (A_3 \times [1,T])$
for $\ell \le 2$. By \eqref{NS1}, $D^\ell u_t \in L^\I (A_{3} \times
[1,T])$ for $\ell \le 1$.

Choose a smooth cutoff function $\zeta(x) \ge 0$ with $\zeta(x)=1$
when $|x| <2\la R_1+5$ and $\zeta(x)=0$ when $|x|>2\la R_1+6$. Let
$w(x,t)$ be a solution supported in $x \in A_{4}$ of
\begin{equation}
\div w(x,t) =u(x,t) \cdot \nb \zeta(x), \quad \norm{ w(\cdot,t)}_{H^2} \lec
\norm{u(\cdot,t) }_{H^1(A_3)}\quad \forall 1\le t\le T,
\end{equation}
by a construction uniform in $t$, see e.g.~\cite[\S III.3]{Galdi1}.
In particular $\div \pd_t w = \pd_t u \cdot \nb \zeta$ and
$\norm{\pd_t w(\cdot,t)}_{H^1} \lec \norm{\pd_t u(\cdot,t)}_{L^2(A_3)}$.  Let
\begin{equation} 
\td u=\zeta u - w .
\end{equation}
 One can check that $\td u$ is a suitable weak solution of
\begin{equation}\label{cutoff.eq}
\pd _t \td u -\De \td u + (\td u \cdot \nb) \td u + \nb \td p =f ,
\quad \div \td u = 0,
\end{equation}
satisfying
\begin{equation}
\td u(x,t)=0 \quad \text{if} \quad |x|>2\la R_1+6;\quad
f(x,t)=0 \quad \text{if} \quad x \not \in A_{4},
\end{equation}
\begin{equation}
\norm{f}_{L^\I_t H^1_x(\R^3 \times (1,T))} \le C(\ga,C_*),
\end{equation}
and by  Lemma \ref{th2.1}
\begin{equation}
\label{th3.1-cutoffbound}
\esssup_{0<t<T} \int_{\R^3} |\td u(x,t)|^2 dx + \int_0^{T} \int_{\R^3}
|\nb \td u|^2 dx\, dt< C(\ga,C_*).
\end{equation}

By \eqref{th3.1-cutoffbound}, there exists a time $t_1 \in [1,3/2]$ such
that $\td u(\cdot,t_1) \in H^1(\R^3)$ with $\norm{\td u(\cdot,t_1)}_{H^1} <
C(\ga,C_*)$. Thus there is $T_2=T_2(\ga,C_*)>0$ and a
strong solution $\hat u(x,t): \R^3 \times [t_1,t_1+T_2) \to \R^3$ of
  \eqref{cutoff.eq} with initial condition $\hat u(x,t_1) = \td u(x,t_1)$.
We may assume $T_2<1$.

By weak-strong uniqueness, we  have $\td u(x,t)=\hat u(x,t)$ for $(x,t)
\in  \R^3 \times [t_1,t_1+T_2)$.

If we take $t_2 = t_1+T_2/4<2$ and $\la_*^2=1+T_2/4$, we have
$[t_2,t_2\la^2] \Subset (t_1,t_1+T_2)$ whenever $t_1\in[1,3/2]$ and
$1<\la \le \la_*$.  All spatial derivatives of $\td u$ are a priori
bounded in $\R^3 \times [t_2,t_2\la^2]$. Because $u$ is discretely
self-similar and agrees with $\td u$ in the region
\eqref{inner-region}, we get $|u(x,t)|\le C$ in the region
\eqref{inner-region}.

We have shown the weaker estimate \eqref{eq3.5}. We now show
\eqref{eq3.1}.  Let $f=- u \ot u$.  By $|(e^{t\De}u_0)(x)| \lec
t^{-1/2}\bka{ t^{-1/2} x}^{-1}$ and \eqref{eq3.5},
\begin{equation}
|f(x,t)| \lec \frac 1{x^2+t}, \quad (x,t) \in \R^4_+.
\end{equation}
By Lemma \ref{th:Stokes} (with $m=0$), we get $|\Phi f(x,t)| \lec
t^{-1/2} \bka{x/\sqrt t}^{-2}$, where $\Phi$ is defined in
\eqref{La-def}. Note $\td v = v- \Phi f$ satisfies the bound
\eqref{eq3.5} and the linear Stokes system in $\R^4_+$ with zero
initial data and zero source. By Lemma \ref{th2.4-new},
$\td v \equiv 0$. Thus we have $v=\Phi f$ and
\eqref{eq3.1}.

(ii) Suppose now $u_0$ is in $C^{1,\be}_{loc}$ and $\norm{\nb
  u_0}_{C^{\be}(\wbar {B(x_0,1/2)})} \le C C_*$ for any $x_0 \in \R^3$
with $1\le |x_0| \le \la$. The vorticity $\om=\curl
u$ satisfies
\begin{equation}
\pd_t \om - \De \om =- \curl \nb \cdot (u \ot u),
\end{equation}
and $\om_0 = \om(\cdot,0)$ satisfies $\norm{\om_0}_{
  C^\be(\wbar{B(x_0,{1/2})})}\le CC_*$.  

Denote $r_k = \frac 19 - \frac k{100}$ and $Q_k = B(x_0,r_k) \times
(0,T_1]$ for $k=0,1,\ldots$.  The previous step shows $u \in
C^\ga_{par}(\wbar{Q_0})$.

Let $\eta(x)=\eta_0(x-x_0)$ where $\eta_0(x)$ is a fixed smooth
cut-off function with $\eta_0(x)=1$ for $|x|<r_2$ and $\eta_0(x)=0$
for $|x|> r_1$.  Decompose $\om =\om_1+\om_2+\om_3$ where
\begin{equation}
  \begin{split} 
\om_1(\cdot ,t)& = -\int_0^t e^{\De(t-s)} [\curl \nb \cdot (u \ot u\eta)
](\cdot,s)ds, \\ 
\om_2(\cdot ,t)&=e^{t\De} (\eta \om(\cdot,0)),
  \\ \om_3 &= \om - \om_1-\om_2.
  \end{split}
\end{equation}
Note $(\pd_t - \De) \om_3 = 0$ in $Q_{2}$ and $\om_3(x,0)=0$ for $x\in
B_{r_2}(x_0)$.  Thus, with possibly a smaller $T_1$, both $\om_2$ and
$\om_3$ are bounded in $C^{\be}_{par}(\wbar{Q_{3}})$ by $C(\be,C_*)$.
By singular integral estimates for heat equation, we have
\begin{equation}
\norm{\om_1}_{L^q(\R^3 \times [0,T_1])} \le C
\norm{u \ot u\eta}_{L^q(\R^3 \times [0,T_1])} \le C(\be,C_*,q)
\end{equation}
for any $1<q<\I$. We take $q=\frac 5{1-\be}$.  Thus $\om=\sum_{i=1}^3 \om_i \in
L^q(Q_{3})$.  By elliptic estimate,
\begin{equation}
\norm{\nb u}_{L^q(Q_{4})} \le C\norm{\curl u}_{L^q(Q_3)} +
C\norm{\div u}_{L^q(Q_3)} + C \norm{u}_{L^q(Q_3)} \le C(\be,C_*).
\end{equation}

We now do a similar decomposition of $\om$ with $\eta_0(x)=1$ for
$|x|< r_6$ and $\eta_0(x)=0$ for $|x|> r_5$.  Again $\om_2$ and
$\om_3$ are in $C^{\be}_{par}(\wbar{Q_{7}})$. Rewrite
\begin{equation}
\nb \cdot (u \ot u\eta) = u \cdot \nb (u \eta) \in L^q(\R^3 \times [0,T_1]).
\end{equation}
By heat potential estimate 
\begin{equation}
[\om_1]_{C^\be_{par}(\R^3 \times [0,T_1])} \le C
\norm{u \cdot \nb (u \eta)}_{L^q(\R^3 \times [0,T_1])} \le C(\be,C_*).
\end{equation}
 Thus $\om=\sum_{i=1}^3 \om_i \in
C^\be_{par}(\wbar {Q_{7}})$. By elliptic estimate,
\begin{equation}
\norm{\nb u}_{L^\I(0,T_1; C^\be(\wbar{B(x_0,r_8)})} \le 
C\norm{\curl u} +
C\norm{\div u} + C \norm{u} \le C(\be,C_*)
\end{equation}
where the middle norms are $C^\be_{par}(\wbar {Q_{7}})$-norms. In
particular we have shown
\begin{equation}
|\nb u(x,t)| \le C(\be,C_*), \quad (1\le |x|\le \la , \quad 0 \le t
\le T_1).
\end{equation}
By the same scaling argument for \eqref{eq3.8}, we get
\begin{equation}
|\nb u(x,t)| \le \frac{C(\be,C_*) }{(\sqrt t + |x|)^2}
\end{equation}
in the sub-paraboloid region \eqref{sub-par}.

We now rewrite
\begin{equation}
v_i(x,t) = \int_0^t \int_{\R^3}S_{ij}(x-y,t-s)g_j(y,s)\,dyds
\end{equation}
with
\begin{equation}
g=- u \cdot \nb u, \quad |g(y,s)|\le \frac{C(\be,C_*) }{(\sqrt s + |y|)^3}.
\end{equation}
Thus for $\ell=0,1$, by \eqref{estimate-T} and change of variables
$x=\sqrt t \td x$, $y=\sqrt t \td y$, $s=t \td s$,
\begin{equation}
|D^\ell v(x,t)| \le \int_0^t \int_{\R^3} (|x-y|+\sqrt{t-s})^{-3-\ell}
(\sqrt s + |y|)^{-3} \,dyds
\end{equation}
\begin{equation}
= t^{-(1+\ell)/2} \int_0^1 \int_{\R^3} (|\frac x{\sqrt t}-y|+\sqrt{1-s})^{-3-\ell}
(\sqrt s + |y|)^{-3} \,dyds.
\end{equation}
By Lemma \ref{th:cal}, we get \eqref{eq3.3}.
\end{proof}

\medskip

We next consider axisymmetric DSS flow with no swirl.

\begin{lemma}
[A priori estimates for axisymmetric DSS flow with no swirl]
\label{th3.2}\ 

(i) For any $1<\la< \I$, $0<\ga<1$, and  $C_*>0$, suppose $u$ is a
forward $\la$-DSS \LL solution of \eqref{NS1} with $\la$-DSS initial
data $u_0\in C^{\ga}_{loc}(\R^3 \bs \{ 0 \})$ that is axisymmetric
with no swirl, DSS with factor $\la$, $\div u_0=0$, and
$\norm{u_0}_{C^\ga(\wbar B_{\la} \bs B_1)} \le C_*$.  Then $u$
satisfies \eqref{eq3.1} with constant $C=C(\la,\ga,C_*)$.
Moreover, $u$ is a mild solution of \eqref{NS1} and \eqref{NS2}.

(ii) If furthermore $\norm{u_0}_{C^{1,\be}(\wbar B_\la \bs B_1)} \le
C_*$ for some $0<\be<1$, then \eqref{eq3.3} hold
with constant $C=C(\la,\be,C_*)$.
\end{lemma}

\begin{proof}
We use the same proof of Lemma \ref{th3.1} (i) until we get a time
$t_1 \in [1,3/2]$, a $T_2=T_2(\ga,C_*)\in (0,1)$ and that $\td u$ agrees
with a strong solution in $[t_1,t_1+T_2)$. For $t_2 = t_1 + T_2/4$, all
  spatial derivatives of $\td u(x, t_2)$ are a priori bounded. Since
  $\td u$ has compact spatial support, we get
\begin{equation}
\norm{\td u(\cdot, t_2)}_{H^2(\R^3)} \le C(\ga,C_*).
\end{equation}
By \cite[Th 1]{LMNP}, 
\begin{equation}
\norm{\td u(\cdot, t)}_{H^1(\R^3)} \le C(\ga,C_*), \quad
\norm{\td u(\cdot, t)}_{H^2(\R^3)} \le C(\ga,C_*,t), 
\end{equation}
for $t_2 \le t \le T=4\la^2$, which contains the interval $[t_2,
  t_2\la^2]$.  Since $u$ is $\la$-DSS and $u$ agrees with $\td u$ in the
region $t \ge |x|^2/R_1^2$ and $1 \le t \le T$, we have shown
the boundedness of $u$ in the same region. The rest of the
proof of Lemma \ref{th3.1} then goes through.
\end{proof}

%==================================================================

\section{Uniqueness of DSS solutions with small data}

When the initial data $u_0$ is small in $L^{3,\I}(\R^3)$, the
existence theorem of \cite{GM, CMP, CP} says that there is a
unique mild solution $\um(x,t)$ in the class
\begin{equation}
\norm{\um(t)}_{BC_w([0,\I); L^{3,\I}(\R^3))} \le C \norm{u_0} _{
  L^{3,\I}(\R^3)}.
\end{equation}
Above $BC_w$ means bounded weak-star continuous $L^{3,\I}$-valued
functions of time. It is also known that $\um$ is a \LL solution, see
\cite{LR}. However, for our application later, we will need uniqueness
in a larger class of solutions.

\begin{lemma}\label{th3.3}
Let $u_0$ satisfy the assumptions of Lemma \ref{th3.1} or Lemma
\ref{th3.2} with $C_*$ sufficiently small. Then any $\la$-DSS \LL solution
$u(t)$ of \eqref{NS1} with initial data $u_0$ must agree with the mild
solution $\um$ constructed by \cite{GM, CMP, CP}.
\end{lemma}

In particular, $u(t)$ is allowed to be large in $BC_w([0,\I);
  L^{3,\I}(\R^3))$.  Such a statement that ``large equals small'' is
  not known for general solutions with small $L^{3,\I}$ data. The
  lemma is only for DSS solutions and relies on the estimates of
  Lemmas  \ref{th3.1} and \ref{th3.2}.

\begin{proof} 
Denote $W(x,t) = \um(x,t)$ and $v(x,t)=u(x,t)-\um(x,t)$. They are both
DSS with factor $\la$ and satisfy \eqref{eq3.1}. Thus
\begin{equation}%\label
| W(x,t)| < \frac {C}{(|x|+\sqrt t)}, \quad
| v(x,t)| < \frac {C\sqrt t}{(|x|+\sqrt t)^{2}},
\end{equation}
in $\R^4_+$.  Note that $v$ satisfies
\begin{equation}
\label{th3.3-eq2}
\pd_tv-\De v+(W+v)\cdot \nb v + v \cdot \nb W+\nb p =0, \quad \div v=0.
\end{equation}
We have $-\De p = \sum _{i,j}\pd_i \pd_j ((W + v)_i v_j + v_i W_j)$
and hence $\norm{p(\cdot,t)}_{L^q(\R^3)} \lec
\norm{\bka{x}^{-3}}_{L^q(\R^3)} < \I$ for $1\le t\le \la^2$ and
$1<q<\I$.

Let $\zeta_R(x)=\zeta(x/R)$ where $ \zeta(x)\ge 0$ is a smooth cut off
function with $\zeta(x)=1$ for $|x|<1$ and $\zeta(x)=0$ for $|x|>2$.
Multiplying \eqref{th3.3-eq2} by $v \zeta_R$ and integrating by parts
over $\R^3 \times [1,\la^2]$, we get
\begin{equation}
\label{eq3.20}
\bkt{ \int_{\R^3} \frac { |v(x,t)|^2}2 \zeta_R(x)\,dx}_{t=1}^{\la^2} +
\int_1^{\la^2} \int_{\R^3} \bke{|\nb v|^2 \zeta_R - v \ot W : (\nb v )
  \zeta_R }\,dxdt =I_R
\end{equation}
where
\begin{equation}
I_R:= \int _1^{\la^2} \int _{\R^3} \frac 12 |v|^2 \De
\zeta_R + (\frac {|v|^2}2(W+v) + (W \cdot v+p)v)  \cdot \nb \zeta_R \,dxdt .    
\end{equation}
By Lemma \ref{th3.1} and $p \in L^\I_t L^q_x$, $I_R$ converges to
$0$ as $R \to \I$. The first term in \eqref{eq3.20} also converges and
it converges to
\begin{equation}
\bkt{ \int_{\R^3}
  \frac { |v(x,t)|^2}2\,dx}_{t=1}^{\la^2} = \frac{\la-1}2\int
|v(x,1)|^2 dx \ge 0.
\end{equation}
We have shown
\begin{equation}
\int_1^{\la^2} \int_{\R^3} \bke{|\nb v|^2 \zeta_R - v \ot W : (\nb v )
  \zeta_R }\,dxdt \le o(1) \quad 
\end{equation}
as $R \to \I$. Since
\begin{equation}
\abs{\iint v \ot W : (\nb v ) \zeta_R \,dxdt} \le \frac 12 \iint |\nb
v|^2 \zeta_R + C \iint |vW|^2 \zeta_R
\end{equation}
and the last term converges as $R \to \I$, we get $\iint |\nb v|^2
\zeta_R \to \iint |\nb v|^2<\I$ as $R \to \I$ and by Hardy inequality
\begin{equation}
\iint |\nb v|^2 \le \iint v \ot W : \nb v \lec C_* \bigg(\iint
\frac{|v|^2}{x^2}\bigg)^{\frac 12} \bigg(\iint |\nb v|^2\bigg)^{\frac 12}
\lec C_*\iint |\nb v|^2.
\end{equation}
Thus when $C_*$ is sufficiently small we get $\int_1^{\la^2} \int
_{\R^3} |\nb v|^2 \,dxdt= 0$. Thus $v \equiv 0$.
\end{proof}

%=====================================================================
\section{Existence of large DSS solutions}

In this subsection we prove Theorems \ref{th1.1} and \ref{th1.2}.  Let
\begin{equation}
U(x,t) =  (e^{t \De} u_0)(x).
\end{equation}
By the assumption on $u_0$, $U$ is $\la$-DSS and
\begin{equation}
|U(x,t)| \le \frac {CC_*}{\sqrt t}\bka{\frac x{\sqrt t}}^{-1} = \frac
{C C_*}{\sqrt {x^2+ 2t}}.
\end{equation}

Introduce a parameter $\sigma \in [0,1]$.
We look for a solution $u(x,t)$ of \eqref{NS1} of the form
\begin{equation}
u(x,t) = \sigma U(x,t) + v(x,t), \quad u(x,0)=\sigma u_0(x).
\end{equation}
The difference $v$ satisfies the nonhomogeneous Stokes system
\eqref{v-eq1}
with
\begin{equation}
\label{F-def}
f=- (\sigma U + v)\ot(  \sigma U+v).
\end{equation}

We expect (which is true at least for small $\si$) that $v$ is $\la$-DSS and
\begin{equation}\label{F-eq}
f(x,t) = \la^3 f(\la x , \la^2 t),\quad \forall(x,t)\in \R^4_+.
\end{equation}
In view of Lemmas \ref{th3.1} and  \ref{th3.2},
we also expect
\begin{equation}
\label{vf-decay}
|v(x,t)| \lec \frac 1{\sqrt t}\bka{\frac x{\sqrt t}}^{-1-\ga}, \quad 
|f(x,t)| \lec \frac 1{ t}\bka{\frac x{\sqrt t}}^{-2} = \frac 1{x^2+2t}.
\end{equation}
The decay rate of $v$ can be improved but it is not needed.

We now set up the framework for the application of the Leray-Schauder
theorem.  Let
\begin{equation}\label{Q-def}
Q=\R^3\times [1,\la^2],
\end{equation}
and the Banach space $X=X(\la)$:
\begin{equation}\label{X-def}
X=\bket{
v\in C(Q; \R^3): \quad
  \begin{aligned}
& \div v=0, \, \norm{v}_X <
  \I,\\
& v(x,1)=\la v(\la x , \la^2 ),\ \forall x \in \R^3
\end{aligned}
 },
\end{equation}
where
\begin{equation}
\label{X-norm}
\norm{v}_X := \sup_{(x,t)\in Q} \bka{x}^{1+\ga} |v(x,t)|.
\end{equation}

For each $v \in X$, we define its DSS extension by
\begin{equation}
Ev(x,t) = \la^k v(\la^k x,
\la^{2k}t), \quad \text{for }(x,t)\in \R^4_+,
\end{equation}
where $k$ is the unique integer so that $1 \le \la^{2k}t < \la^2$.

We now define an operator $K: X \times [0,1] \to X$ by
\begin{equation}
K(v,\sigma) := -\Phi[(\sigma U + Ev)\ot( \sigma U+Ev)]|_Q,
\quad \forall v \in X,\ \forall \sigma \in [0,1].
\end{equation}
Above $\Phi$ is defined by \eqref{La-def}.

Note that for $v\in X$ with $\norm{v}_X<M$ and $0 \le \si \le 1$, the
force $f= -(\sigma U + Ev)\ot( \sigma U+Ev)$ satisfies \eqref{F-eq}
and \eqref{vf-decay}, and $\Phi(f)$ defined by \eqref{La-def}
satisfies \eqref{v-eq1} and is $\la$-DSS. By Lemma \ref{th:Stokes},
its restriction to $Q$, $K(v,\sigma)=-\Phi(f)|_{Q}$, is inside $X$ and
$\norm{K(v,\sigma)}_X \le C (C_*+M)^2$.  Thus $K$ indeed maps bounded
sets in $ X \times [0,1]$ into bounded sets in $X$.

Furthermore, $K$ is compact because its main term $\Phi( \sigma U \ot
\sigma U)|_Q$ is one dimensional while the other terms of $K$ have
extra decay by Lemma \ref{th:Stokes} and are H\"older continuous in
$x$ and $t$ by Lemma \ref{th:Holder}.

We have now a fixed point problem
\begin{equation}
\label{LS-eq}
v = K(v,\sigma) \quad \text{in }X
\end{equation}
that satisfies the following:
\begin{enumerate}
\item $K(v,\si): X \times [0,1]\to X$ is compact by the previous
  discussion,
\item it is uniquely solvable in $X$ for small $\sigma$ by
  \cite{GM,CMP,CP} and Lemma \ref{th3.3}, thus the Leray-Schauder
  degree is nonzero, and
\item we have a priori estimate in $X$ for solutions $v$ of
  \eqref{LS-eq} uniformly for all $\sigma\in [0,1]$ by Lemma
  \ref{th3.1} or \ref{th3.2}.
\end{enumerate}

By Leray-Schauder degree theorem (see e.g.~\cite{Mawhin}), there is a
solution $v \in X$ of \eqref{LS-eq} with $\sigma=1$. It follows that
$Ev$ satisfies \eqref{v-eq1} with $f=-(U+Ev)\ot (U+Ev)$, and hence
$u=U+Ev$ is a $\la$-DSS \LL solution of \eqref{NS1} with initial data
$u_0$.  \myendproof

\section*{Acknowledgments}
I would like to thank Dong Li with whom I had many fruitful
discussions on \cite{JS}. I also thank Hideyuki Miura for discussions
on Lemma \ref{th3.3} for general solutions that may not be DSS.  This
work was partially supported by the Natural Sciences and Engineering
Research Council of Canada grant 261356-08.

%=====================================================================

{\small
Department of Mathematics,
University of British Columbia,  
Vancouver, BC V6M1Z2, Canada

e-mail: ttsai@math.ubc.ca}

\end{document}